\documentclass{daj}

\dajAUTHORdetails{
  title = {Counting Rational Points on Quadric Surfaces}, 
  author = {Tim Browning, and Roger Heath-Brown},
  plaintextauthor = {Tim Browning, and Roger Heath-Brown},
}

\dajEDITORdetails{%
   year={2018},
   volume={},
   number={15},
   received={4 January 2018},   
   published={7 September 2018},  
   doi={10.19086/da.4375},       
}   

\usepackage{amsfonts, amsthm, amsmath, amssymb}

\RequirePackage{mathrsfs} \let\mathcal\mathscr
\numberwithin{equation}{section}

\newcommand{\beql}[1]{\begin{equation}\label{#1}}
\newcommand{\eeq}{\end{equation}}
\newcommand{\twosum}[2]{\sum_{\substack{#1\\#2}}}
\newcommand{\be}{\mathbf{e}}

\newtheorem{theorem}{Theorem}[section]
\newtheorem{lemma}[theorem]{Lemma}

\newtheorem{conjecture}[theorem]{Conjecture}

\theoremstyle{definition}

\renewcommand{\phi}{\varphi}
\renewcommand{\rho}{\varrho}

\newcommand{\card}{\#}
\newcommand{\0}{\mathbf{0}}
\newcommand{\PP}{\mathbb{P}}

\newcommand{\FF}{\mathbb{F}}
\newcommand{\ZZ}{\mathbb{Z}}
\newcommand{\ZZp}{\mathbb{Z}_{\mathrm{prim}}}
\newcommand{\NN}{\mathbb{N}}
\newcommand{\QQ}{\mathbb{Q}}
\newcommand{\RR}{\mathbb{R}}

\newcommand{\cR}{\mathcal{R}}

\renewcommand{\leq}{\leqslant}
\renewcommand{\le}{\leqslant}
\renewcommand{\geq}{\geqslant}
\renewcommand{\ge}{\geqslant}
\renewcommand{\bar}{\overline}

\newcommand{\ma}{\mathbf}

\newcommand{\M}{\mathbf{M}}
\newcommand{\x}{\mathbf{x}}
\newcommand{\y}{\mathbf{y}}
\renewcommand{\c}{\mathbf{c}}

\renewcommand{\u}{\mathbf{u}}

\newcommand{\ee}{\mathbf{e}}

\newcommand{\ve}{\varepsilon}

\DeclareMathOperator{\rank}{rank}

\DeclareMathOperator{\meas}{meas}

\DeclareMathOperator{\Diag}{Diag}

\newcommand{\modd}[1]{\; ( \text{mod} \; #1)}

\newcommand{\Dbad}{\Delta_{\mathrm{bad}}}

\newcommand{\1}{\mathbf{1}}

\begin{document}

\begin{frontmatter}[classification=text]

\title{Counting Rational Points on Quadric Surfaces}
\author[tdb]{T.D.\ Browning\thanks{Supported by EPSRC grant \texttt{EP/P026710/1}
}}
\author[rhb]{D.R.\ Heath-Brown}

\begin{abstract}
We give an upper bound for the number of rational points of 
height at most $B$, lying on a surface defined by a quadratic
form $Q$.  The bound shows an explicit dependence on $Q$.  It is
optimal with respect to $B$, and is also
optimal for typical forms $Q$.
\end{abstract}
\end{frontmatter}


\section{Introduction}

Let $Q \in \ZZ[x_1,x_2,x_3,x_4]$ be a non-singular   
quadratic form, with 
height $\|Q\|$ and discriminant $\Delta_Q$.
We shall be concerned with completely uniform estimates for the number
of rational points of bounded height lying on the projective quadric surface 
$Q=0$.
For any  $B \geq 1$ we define the counting function
$$
N(B) = \#\{\mathbf{x}\in \ZZp^4: Q(\mathbf{x})=0, ~|\x|\leq B\},
$$
where $|\x|=\max_{1\leq i\leq 4}|x_i|$. 
Our upper bound for $N(B)$ will depend on $\Delta_Q, \|Q\|$ and on the 
square-full part 
\[\Dbad = \prod_{\substack{p^e\| \Delta_Q\\ e\geq 2}} p^e\]
of the discriminant.
It  will also be convenient to introduce the arithmetic function
\begin{equation}\label{eq:phid}
\varpi(m)=\prod_{p\mid m}(1+p^{-1}).
\end{equation}

The following is our main result.

\begin{theorem}\label{t:upper}
 Let  $\chi$ denote  the Dirichlet character induced by
 the Legendre symbol $(\frac{\Delta_Q}{\cdot})$, and assume that
 $\Dbad\le B^{1/20}$. Then for any fixed $\ve>0$ we have
\[N(B)\ll_{\ve} \varpi(\Delta_Q)\Dbad^{1/4+\ve}
\left(\frac{\|Q\|^4}{|\Delta_Q|}\right)^{5/8}
\Pi_B\left(B^{4/3}+\frac{B^2}{|\Delta_Q|^{1/4}}\right),\]
where
\begin{equation}\label{eq:cQ}
\Pi_B=\prod_{p\le B}\left(1+\frac{\chi(p)}{p}\right).
\end{equation}
The implied constant in this estimate only depends on the choice of $\ve$. 
\end{theorem}
\bigskip

The theorem is a refinement of work by Browning \cite{qpub} in three
key aspects. Firstly, the latter has a $B^\ve$-loss; secondly, it only
pertains to the case of diagonal quadratic forms $Q$; and thirdly, it
requires that $\Delta_Q$ is square-free.  
Although Theorem~\ref{t:upper} handles general quadratic forms, it is still
sharpest for quadratic forms  
whose discriminant is close to being square-free and
$\|Q\|^4$ in size.

For a fixed form $Q$ with at least one non-trivial zero one can deduce
from the results of Heath-Brown \cite[Theorems 6 \& 7]{circle} that
$$
N(B)\sim 
\begin{cases}
c_QB^2, & \text{ if $\Delta_Q\not=\square$},\\
c_QB^2\log B, &\text{ if $\Delta_Q=\square$},
\end{cases}
$$
as $B\rightarrow\infty$, where $c_Q$ is a positive constant.  When
$\Delta_Q\not=1$ is square-free and of order $\|Q\|^4$,
the constant $c_Q$
is of exact order $|\Delta_Q|^{-1/4}L(1,\chi)$, so that Theorem \ref{t:upper}
is optimal for large $B$, 
apart possibly for the factors $\varpi(\Delta_Q)$,
$\Dbad^{1/4+\ve}$ and $(\|Q\|^4/\|\Delta_Q|)^{5/8}$.

It is natural to ask to what extent one can produce uniform upper bounds
for $N(B)$ which depend only on $B$ and not on the  coefficients of $Q$.
In the spirit of recent work by Walsh \cite{walsh} on rational curves, we
have been led to make the following conjecture.

\begin{conjecture}
There is  an absolute constant $c>0$ such that 
$$
N(B)\leq \begin{cases}
cB^2, & \text{ if $\Delta_Q\neq \square$,}\\
cB^2\log B, & \text{ if $\Delta_Q\neq 0$,}
\end{cases}
$$
for every $B\ge 2$. 
\end{conjecture}

It might seem that the occurrence of the factors $\Dbad$ and
$\|Q\|^4/|\Delta_Q|$ is a defect of Theorem \ref{t:upper}.
However we will show
below that if an estimate of the type
\beql{abb}
N(B)\ll \varpi(\Delta_Q)\Dbad^{\alpha}
\left(\frac{\|Q\|^4}{|\Delta_Q|}\right)^{\beta}
\Pi_B\left(B^{4/3}+\frac{B^2}{|\Delta_Q|^{1/4}}\right),
\eeq
holds, with constants $\alpha$ and $\beta$, then we must have
$\alpha\ge 1/4$. However it is not clear whether a power of
$\|Q\|^4/|\Delta_Q|$ is necessary.
In concurrent  work \cite{xiyi} we have  applied Theorem \ref{t:upper} to
investigate the density of rational points on the hypersurface
\[x_0y_0^2+x_1y_1^2+x_2y_2^2+x_3y_3^2=0\]
in $\mathbb{P}^3\times\mathbb{P}^3$, and for this it is essential
that $\alpha<1/2$ and $\beta<3/4$.

To show that one must have $\alpha\ge 1/4$ we use the form
\[Q(\x)=k(x_1^2+x_2^2+x_3^2-x_4^2)\]
with $k\in\NN$.  One easily sees that $N(B)\gg B^2$, while
$\Dbad=k^4$ and
\[\varpi(\Delta_Q)\left(\frac{\|Q\|^4}{|\Delta_Q|}\right)^{\beta}
\Pi_B\left(B^{4/3}+\frac{B^2}{|\Delta_Q|^{1/4}}\right)\ll
\varpi(k)^2\left(B^{4/3}+\frac{B^2}{k}\right).\]
Thus for (\ref{abb}) to hold one must have $\alpha\ge 1/4$.

A few words are  in order about the size of the factor $\Pi_B$. 
We always have $\Pi_B=O(\log B)$ and this is the true order of $\Pi_B$ when
$\Delta_Q=\square$. Suppose now that   $\Delta_Q\neq \square$ and note first that 
\begin{equation}\label{eq:gower}
\Pi_B\ll\exp\left(\sum_{p\le B}\frac{\chi(p)}{p}\right).
\end{equation}  
However, with
$\sigma=1+(\log B)^{-1}$, we have
  \begin{align*}
    \sum_{p\le B}\frac{\chi(p)}{p}&=
    \sum_{p\le B}\frac{\chi(p)}{p^{\sigma}}+O(1)\\
    & =\sum_p\frac{\chi(p)}{p^{\sigma}}+O(1)\\
    &=\log L(\sigma,\chi)+O(1),
    \end{align*}
  the final sum running over all primes $p$. This shows that 
\[\Pi_B\ll L\left(1+\frac{1}{\log B},\chi\right).\]
In fact it is possible to show that 
$\Pi_B$ is bounded independently of $B$.
To see this, 
a standard argument found at the end of Chapter~7 of Davenport
\cite{davenport}  shows that there is a constant $c(\Delta_Q)>0$ such that 
$$
\left|\sum_{p\leq B} \frac{\chi(p)\log p}{p} \right|\leq c(\Delta_Q).
$$
(One actually finds that 
$c(\Delta_Q)\ll
1+ |L(1,\chi)|^{-1}\{|L'(1,\chi)|+\sqrt{|\Delta_Q|}\log |\Delta_Q|\}
$  is admissible, by  invoking the P\'olya--Vinogradov inequality in the argument.) 
This can be  combined with partial summation in \eqref{eq:gower} to yield the claim.

The case in which $\Delta_Q$ is a square is rather different from the
generic situation, not least because  $\Pi_B$ then has
order $\log B$. For the bulk of the paper we will consider only the
situation in which $\Delta_Q\not=\square$.  We will then point out the
modifications necessary to handle the alternative case in the final
section. 
 
Our strategy for the proof uses $O(B^{4/3})$ plane slices
through the region $|\x|\le B$. Each slice produces a conic, and we
estimate the number of points on each of these individually. This
procedure naturally gives a bound which is $\gg B^{4/3}$. The bound
for an individual conic is somewhat complicated, and the procedure by
which we average over the various plane slices is correspondingly
delicate.  In particular much care is necessary if one is to avoid
extraneous factors of the type $\log B$.

\section{Preliminary steps}\label{sP}

\subsection{Geometry of numbers}

We begin by recording a version of Siegel's lemma. (See 
\cite[Lemma 1(iv)]{annal}, for example.)

\begin{lemma}
Let  $\mathbf{x}\in \ZZ^4$ such that $|\x|\leq B$.
Then there exists a  vector $\c \in \ZZp^4$ with $|\c|\ll B^{1/3}$, 
such that $\mathbf{x}.\c=0$.
\end{lemma}

It follows that
\begin{equation}\label{eq:step1}
N(B)\leq \sum_{\substack{\c\in \ZZp^4\\ |\c|\ll B^{1/3}}} 
\#\left\{ \x\in\ZZp^4:\,\x.\c=0,\,Q(\x)=0,\,|\x|\le B\right\}.
\end{equation}
We write $\be_4=|\c|^{-1}\c$ and extend to an orthonormal basis
$\be_1,\be_2,\be_3,\be_4$ of $\RR^4$. We may of course choose
$\be_1,\be_2,\be_3$ so that the matrix of $Q$ with respect to the
basis is
\beql{mx}
\mathbf{U}^T\mathbf{M}\mathbf{U}= 
\begin{pmatrix}
  \mu_1 & 0 & 0 & a\\
  0 & \mu_2 & 0 & b\\
  0 & 0 & \mu_3 & c\\
  a & b & c & d
\end{pmatrix}
\eeq
say, where $\mathbf{M}$
is the matrix associated to  $Q$, and $\mathbf{U}$ is the orthogonal
matrix with columns
$\be_1,\be_2,\be_3,\be_4$. Indeed we may suppose that
\[|\mu_3|\le|\mu_2|\le|\mu_1|\ll \|Q\|.\]
We can interpret the above representation as saying that the quadratic
form $Q$, when restricted to the plane $\mathbf{x}.\c=0$, can be
diagonalized as $\Diag(\mu_1,\mu_2,\mu_3)$.  Our goal is to use
information about the size of $\mu_1,\mu_2,\mu_3$ to restrict the
region in which $\x$ can lie.  We will establish the following result,
which involves the dual form $Q^*$, with  underlying matrix
$\mathbf{M}^{\text{adj}}=\Delta_Q \mathbf{M}^{-1}$.

\begin{lemma}\label{Q*B}
  Let $\c\in\ZZp^4$ be given, with $Q^*(\c)\not=0$.  Then there are ellipsoids
  $E_0,\ldots,E_m$ with
    \[m\ll \log\left(2+\frac{|\c|^2\|Q\|^3}{|Q^*(\c)|}\right),\] 
  such that each $E_j$ is centred at the origin and has
  \[\meas(E_j)\ll \frac{|Q^*(\c)|.\|Q\|^3 B^3}{|\c|^2|\Delta_Q|^{3/2}},\]
  and so that
\[\{\x\in\RR^4:\,Q(\x)=\x.\c=0,\,|\x|\le B\}\subset\bigcup_{j=0}^m E_j.\]
 \end{lemma}

\begin{proof} 
The matrix (\ref{mx}) must have entries
which are $O(\|Q\|)$, since the entries of $\mathbf{U}$ have modulus
at most 1.  It therefore follows that
\beql{B1}
|\Delta_Q|\ll \|Q\|^2|\mu_1\mu_2|.
\eeq
The adjoint of the matrix $\mathbf{U}^T\mathbf{M}\mathbf{U}$ will have
$\mu_1\mu_2\mu_3$ as its bottom right entry, whence
$\mathbf{U}^T\mathbf{M}^{-1}\mathbf{U}$ will have
$\det(M)^{-1}\mu_1\mu_2\mu_3$ as its bottom right entry. It follows that
\[(0,0,0,1)\mathbf{U}^T\mathbf{M}^{-1}\mathbf{U}
\begin{pmatrix} 0\\ 0 \\0 \\1 \end{pmatrix}= 
\det(M)^{-1}\mu_1\mu_2\mu_3.\]
However
\[\mathbf{U}
\begin{pmatrix}0\\ 0 \\0 \\1 \end{pmatrix}=\be_4,\]
whence
\[\be_4^T\mathbf{M}^{-1}\be_4=\det(M)^{-1}\mu_1\mu_2\mu_3.\]
We then conclude that
\beql{B2}
\mu_1\mu_2\mu_3=Q^*(\be_4).
\eeq

If $\x.\c=0$ with $|\x|\le B$, then we can write
$\x=y_1\be_1+y_2\be_2+y_3\be_3$, whence
\[Q(\x)=\mu_1y_1^2+\mu_2y_2^2+\mu_3y_3^2.\]
Moreover $|y_i|\le|\x|\le B$, since the vectors $\be_i$ were taken to
be orthonormal.  Thus if $Q(\x)=0$ we have
\[|\mu_1y_1^2+\mu_2y_2^2|\le |\mu_3|y_3^2\le |\mu_3|B^2.\]
When $\mu_1$ and $\mu_2$ have the same sign we immediately deduce that
$(y_1,y_2,y_3)$ lies in a 3-dimensional ellipsoid $E_0$  
 having semi-axes of lengths $2\sqrt{|\mu_3/\mu_1|}B$, $2\sqrt{|\mu_3/\mu_2|}B$
and $2B$. Thus, using (\ref{B1}) and (\ref{B2}) we have
\beql{B5}
\meas(E_0) 
\ll\frac{|\mu_3|}{\sqrt{|\mu_1 \mu_2|}}B^3
=\frac{|Q^*(\be_4)|}{|\mu_1 \mu_2|^{3/2}}B^3
\ll \frac{|Q^*(\c)|.\|Q\|^3 B^3}{|\c|^2|\Delta_Q|^{3/2}},
\eeq
since we took $\be_4=|\c|^{-1}\c$. This means of course that $\x$ is
also restricted to lie in such an ellipsoid.

When $\mu_1$ and $\mu_2$ have opposite signs things are a little more
awkward. Let $\nu=\sqrt{-\mu_2/\mu_1}$.  Then if $Q(\x)=0$ as above we
have
\beql{B3}
|y_1^2-\nu^2y_2^2|\le |\mu_3/\mu_1|B^2.
\eeq
Suppose, say that $y_1$ and $y_2$ are both non-negative (the other
cases being handled similarly).  Then if
\[y_1+\nu y_2\le \sqrt{|\mu_3/\mu_1|}B\]
we see that $(y_1,y_2,y_3)$ lies in an ellipsoid $E_0$ with semi-axes
whose lengths are 
$$
2\sqrt{|\mu_3/\mu_1|}B, \quad 
2\nu^{-1}\sqrt{|\mu_3/\mu_1|}B=\sqrt{|\mu_3/\mu_2|}B, 
\quad \text{ and } \quad 2B, 
$$
as
before.  Otherwise 
\beql{B4}
2^{m-1}\sqrt{|\mu_3/\mu_1|}B<y_1+\nu y_2\le 2^m\sqrt{|\mu_3/\mu_1|}B
\eeq
for some positive integer $m$. It follows from (\ref{B3}) that
\[y_1^2\le \nu^2B^2+|\mu_3/\mu_1|B^2\le 2\nu^2B^2,\]
and hence $y_1\le 2\nu B$ and $y_1+\nu y_2\le 3\nu B$. We therefore
have $2^m\le 6\sqrt{|\mu_2/\mu_3|}$, so that
\[ m\ll 1+\log\left|\frac{\mu_2}{\mu_3}\right|=
1+\log\left|\frac{\mu_1\mu_2^2}{\mu_1\mu_2\mu_3}\right|
\ll \log\left(2+\|Q\|^3/|Q^*(\be_4)|\right).\] 

For each such $m$ we have
\[|y_1-\nu y_2|=\frac{|y_1^2-\nu^2y_2^2|}{y_1+\nu y_2}
\le \frac{|\mu_3/\mu_1|B^2}{2^{m-1}\sqrt{|\mu_3/\mu_1|}B}=
2^{1-m}\sqrt{|\mu_3/\mu_1|}B.\]
Since
\[|y_1+\nu y_2|\le 2^m\sqrt{|\mu_3/\mu_1|}B,\]
by (\ref{B4}), the point $(y_1,y_2)$ lies in a parallelogram of area
\[\ll \nu^{-1}2^{1-m}\sqrt{|\mu_3/\mu_1|}B\times 2^m\sqrt{|\mu_3/\mu_1|}B\ll
\frac{|\mu_3|}{\sqrt{|\mu_1 \mu_2|}}B^2.\]
It follows that, for each $m$, there is an ellipse of area
$O(|\mu_3|.|\mu_1\mu_2|^{-1/2}B^2)$ containing $(y_1,y_2)$.  We then
get 3-dimensional ellipsoids $E_m$, one for each $m$, with volume
bounded as in (\ref{B5}), such that $(y_1,y_2,y_3)$ necessarily lies
in one of the $E_m$.  This completes the proof of the lemma.
\end{proof}

The following result is well-known in principle, but merits a formal
proof.

\begin{lemma}\label{lem:dav}
  Let $\mathsf{\Lambda}\subseteq\RR^n$ be a lattice of dimension $k\leq n$.
  Then there exists a basis $\mathbf{g}^{(1)},\ldots,\mathbf{g}^{(k)}$
  of $\mathsf{\Lambda}$ for which
  \beql{every}
  \prod_{j=1}^k |\mathbf{g}^{(j)}|\ge\det(\mathsf{\Lambda}),
  \eeq
  and such that if $\x\in\RR^n$ can be written as 
  \beql{xr}
  \x=\sum_{j=1}^k c_{j}\mathbf{g}^{(j)},
  \eeq
then 
\[|c_j|\le n^{2n}|\x|/|\mathbf{g}^{(j)}|.\]
\end{lemma}

The constant $n^{2n}$ is certainly not optimal, but that is 
not important for us. 

\begin{proof}[Proof of Lemma \ref{lem:dav}]
The statement \eqref{every} clearly holds for any basis of
$\mathsf{\Lambda}$.  
For the remaining fact we appeal to Cassels' treatise on the  
geometry of numbers \cite{cassels}. This has the deficiency of
only  applying to lattices of full rank. 
Thus we content ourselves here with giving
a detailed proof when $k=n$, leaving to the reader
the necessary modifications required to handle $k<n$.

According to the corollary on page 222 of Cassels \cite{cassels}, if
the successive minima of $\mathsf{\Lambda}$ are
\[\lambda_1\le\dots\le\lambda_n,\]
then we may choose a basis 
$\mathbf{g}^{(1)},\ldots,\mathbf{g}^{(n)}$
of $\mathsf{\Lambda}$ so that
\[|\mathbf{g}^{(j)}|
\begin{cases} 
=\lambda_1, & \text{ if $j=1$},\\
\le \tfrac{1}{2}j\lambda_j, & \text{ if $j\ge 2$.}  
\end{cases}
\]
In particular we have $|\mathbf{g}^{(j)}| \le n\lambda_j$.
Let $\mathsf{\Lambda}_j$ be the $(n-1)$-dimensional lattice with basis 
 \[\mathbf{g}^{(1)},\ldots,\mathbf{g}^{(j-1)},\mathbf{g}^{(j+1)},
  \ldots,\mathbf{g}^{(n)},\]
  and let $V_j$ be the corresponding vector
  space over $\RR$. Then a consideration of the respective
  fundamental parallelepipeds shows that
  \[\det(\mathsf{\Lambda})=
  \det(\mathsf{\Lambda}_j){\rm dist}(\mathbf{g}^{(j)},V_j).\]
However
\[\det(\mathsf{\Lambda}_j)\le\prod_{\substack{i=1\\ i\not=j}}^n|\mathbf{g}^{(i)}|,\]
while
\[\det(\mathsf{\Lambda})\ge 2^{-n}{\rm Vol}_n\prod_{i=1}^n\lambda_i,\]
by Theorem V on page 218 of Cassels \cite{cassels}, where ${\rm Vol}_n$
is the volume of the unit ball in $\RR^n$.  By comparison with the
region $\sum|x_i|\le 1$ we have ${\rm Vol}_n\ge 2^n/n!\ge 2^nn^{-n}$. Thus
\[{\rm dist}(\mathbf{g}^{(j)},V_j)\ge
\frac{n^{-n}|\mathbf{g}^{(j)}|\lambda_1\ldots\lambda_n}
     {|\mathbf{g}^{(1)}|\ldots|\mathbf{g}^{(n)}|}\ge
     \frac{|\mathbf{g}^{(j)}|}{n^{2n}}.\]
However if $\x$ is represented as in (\ref{xr}), then
\[\frac{|\x|}{|c_j|}\ge{\rm dist}(\mathbf{g}^{(j)},V_j),\]
so that
\[|c_j|\le n^{2n}|\x|/|\mathbf{g}^{(j)}|,\]
  as claimed.
  \end{proof}

Using the previous lemma we now have the following.
\begin{lemma}\label{lattice}
  Let $V\subseteq\RR^4$ be a 3-dimensional vector space, and let
  $\mathsf{\Lambda}\subseteq V$ be a 3-dimensional lattice. Suppose that
   $E$ is an ellipsoid in $V$, centred on the origin.  Then there
  exists a basis $\mathbf{f}^{(1)},\mathbf{f}^{(2)},\mathbf{f}^{(3)}$
  of $\mathsf{\Lambda}$ and positive numbers $L_1,L_2,L_3$ with
  \[L_1L_2L_3\ll\frac{\meas(E)}{\det(\mathsf{\Lambda})},\]
  such that if one writes $\x\in\mathsf{\Lambda}\cap E$ as
  $\x=\sum_{j}\lambda_{j}\mathbf{f}^{(j)}$, then 
$|\lambda_{j}|\le L_j$.
\end{lemma}

\begin{proof}
  Let $\mathbf{e}\in\RR^4$ be a unit vector orthogonal to $V$, and let
  $\mathbf{A}\in \mathrm{GL}_4(\RR)$ 
  be chosen to fix 
  $\mathbf{e}$    and $V$, and to map
  $E$ to the unit 3-dimensional ball in $V$.
Thus
\beql{R3}
1\ll |\det(\mathbf{A})|\meas(E)\ll 1.
\eeq
Moreover $\mathbf{A}\mathsf{\Lambda}$ is a lattice of determinant
$|\det(\mathbf{A})|\det(\mathsf{\Lambda})$. We now wish to apply
Lemma \ref{lem:dav} to the 
3-dimensional lattice $\mathbf{A}\mathsf{\Lambda}$
in $\RR^4$. According to the
lemma we see that there is a basis
$\mathbf{g}^{(1)},\mathbf{g}^{(2)},\mathbf{g}^{(3)}$ such that, if
$\y=\sum_{j}\lambda_{j}\mathbf{g}^{(j)}$ then $|\lambda_j|\le
L_j|\y|$, with $L_j=4^8/|\mathbf{g}^{(j)}|$ for $1\le j\le
3$. In particular, if $\y$ is in the unit ball, then $|\lambda_j|\le L_j$.

We also see that the values  $L_j$ satisfy
\[L_1L_2L_3\ll\prod_{j=1}^3 |\mathbf{g}^{(j)}|^{-1}\ll
\det(\mathbf{A}\mathsf{\Lambda})^{-1}\ll
\frac{\meas(E)}{\det(\mathsf{\Lambda})},\]
by (\ref{every}) and (\ref{R3}).

Since $\mathbf{g}^{(j)}\in\mathbf{A}\mathsf{\Lambda}$ we may write
$\mathbf{g}^{(j)}=\mathbf{A}\mathbf{f}^{(j)}$, with
$\mathbf{f}^{(j)}\in\mathsf{\Lambda}$. Indeed we see that
$\mathbf{f}^{(1)},\mathbf{f}^{(2)},\mathbf{f}^{(3)}$ form a basis
of $\mathsf{\Lambda}$.  Moreover, if
$\x=\sum_{j}\lambda_{j}\mathbf{f}^{(j)}$, we find that
$\mathbf{A}\x=\sum_{j}\lambda_{j}\mathbf{g}^{(j)}$. When $\x\in E$ the
vector $\y=\mathbf{A}\x$ will lie in the unit ball, and we may
conclude that $|\lambda_j|\le L_j$, as required.
\end{proof}

\subsection{Conics}

Our treatment of the cardinality in \eqref{eq:step1} relies on a
general estimate for the number of rational points of bounded height
on conics.   

The first ingredient in this is the following result.
\begin{lemma}\label{V3}
Let $q(x_1,x_2,x_3)$ be a non-singular integral quadratic form.  Let
$L_1, L_2,L_3>0$.  Then
there are $O(1+(L_1L_2L_3)^{1/3})$ primitive integer solutions to 
$q(x_1,x_2,x_3)=0$ satisfying $|x_i|\le L_i$ for $1\le i\le 3$.
  \end{lemma}
This is basically Lemma 6 of the authors' paper \cite{n-2}, in which
one assumes that the $L_i$ are all at least 1. When $L_3<1$, say,
the points are restricted to a line so that there are at most two
primitive solutions.

For the second ingredient, let $q$ be a 
non-singular ternary quadratic form defined over $\ZZ$ as above,
with discriminant $\Delta_q$. 
For any prime $p$ we let $\bar q$ denote the reduction of $q$ modulo $p$.
We define a completely multiplicative function 
$\chi_q: \NN\to \{0,\pm 1\}$, via 
$$
\chi_q(p)=\begin{cases}
+1, & \text{ if $\rank \bar q=2$ and $\bar q$ is reducible over $\FF_p$,} \\
-1, & \text{ if $\rank \bar q=2$ and $\bar q$ is irreducible over $\FF_p$,} \\
0, & \text{ if $\rank \bar q\not= 2$.} 
\end{cases}
$$
For any non-zero integer $M$, let $M^\square=\prod_{p^e\| M, e\geq 2}p^e$
denote the (positive)
square-full part of $M$ (so that $\Dbad= \Delta_Q^\square$, 
for example).  With this notation
the following result draws together a number of arguments that appear
in the literature and has the advantage of automatically detecting
when the quadratic form is isotropic over $\QQ$. 

\begin{lemma}\label{two}
Let $q$ be a non-singular ternary quadratic form over $\ZZ$ with matrix
$\mathbf{A}$.  Let  $\Delta_q=\det \mathbf{A}$ and let 
 $D(q)$ be the  highest
common factor of the $2\times 2$ minors of $\mathbf{A}$. Then there are
lattices $\mathsf{\Lambda}_i$ for $1\le i\le I$ such that
\[\left\{\y \in \ZZp^3: q(\y)=0\right\}
\subseteq\bigcup_{i=1}^I\mathsf{\Lambda}_i.\]
Moreover we have
\beql{NL}
\det(\mathsf{\Lambda}_i)\gg \frac{|\Delta_q|}{(D(q)^\square)^{3/2}}
\eeq 
  for all $i$, and $I\ll C(q)$,   where
\[ C(q)=\prod_{\substack{p^\xi \| \Delta_q\\ \text{$p \mid  2 D(q)$}}}
\tau(p^\xi ) \prod_{\substack{p^\xi  \| \Delta_q\\ p\nmid 2D(q)}}
\left\{\sum_{k=0}^\xi \chi_q(p^k)\right\}.\]
In particular it may happen that $C(q)=0$, in which case
$q(\y)=0$ has no solutions in $\ZZp^3$.
\end{lemma}

\begin{proof}
A statement of this sort follows from 
\cite[Lemma ~2.4]{dp4-fibration}
except that one would have $D(q)$ in place $D(q)^\square$ in \eqref{NL}.
To show that the dependence on $D(q)$ can be weakened in
the way that is claimed here one merely applies the argument used in
\cite[Lemma 5]{qpub}. We briefly recall the necessary modifications
for completeness.  
Following  the  treatment in
\cite[Cor.~2]{n-2} and 
 \cite[Thm.~2]{cubic}, the idea is to consider the congruence
 conditions imposed on primitive integer solutions to $q(\y)=0$, in
 order to show that the solutions in which we are interested lie on a
 small number  
of lattices with large determinant.  Suppose that $p^\beta\| D(q)$ and
$p^\xi\| \Delta_q$ with $0\leq \beta\leq\xi$.  
According to the proof of  \cite[Thm.~2]{cubic}, 
the points in which we are interested lie on a union of at most
$c_p^{(1)} \tau(p^\xi)$ lattices, each of determinant  
 $c_p^{(2)}p^{\xi-[3\beta/2]}$, for absolute constants $c_p^{(i)}$ such that
$c_p^{(i)}=1$ for $p>2$.
This is satisfactory for $p=2$, and also when $p>2$ and $\beta\ge 2$ 
so that we only need to refine
the statement when  $p>2$ and $\beta\leq 1$. 

On diagonalising $q$ over the ring $\ZZ/p^\xi \ZZ$ we may suppose
without loss of generality that  $\y\in \ZZp^3$ satisfies the
congruence 
\begin{equation}\label{eq:congruence}
 Ay_1^2+p^\beta By_2^2+p^\gamma C y_3^2\equiv 0\bmod{p^\xi},
\end{equation}
for  $A,B,C\in \ZZ$ such that $p\nmid ABC$,  and 
where $\beta\leq \gamma$ and $\beta+\gamma=\xi$. 
Suppose first that $\beta=0$ and note that 
$\chi_q(p)=(\frac{-AB}{p})$. If $\chi_q(p)=1$ we don't need to do
anything new. If  
$\chi_q(p)=-1$, on the other hand,  we easily see there are no
primitive integer solutions if $2\nmid \xi$, while if $2\mid \xi$ the
points lie on a unique lattice of determinant $p^\xi.$ 
Suppose next that $\beta=1$, so that $\gamma=\xi-1$.
We claim that the points in which we are interested in lie on one of
at most $2$ lattices, each of determinant $p^{\xi}$. Suppose that
$\xi=2k$ is even, with $k\geq 1$. 
Then the  congruence \eqref{eq:congruence} can be used to deduce that 
$p^{k}\mid y_1$ and $p^{k-1}\mid y_2$. A change of variables then leads to a 
congruence of the form
$
Bz_2^2+ C z_3^2\equiv 0\bmod{p},
$
This final congruence forces $\y$ to lie on a union of at most $2$ lattices,  
each of determinant $p^{k}\cdot p^{k-1}\cdot p=p^\xi$. The case in
which $\xi$ is odd is similar. 
\end{proof}
\medskip

Let us now consider the effect of this in \eqref{eq:step1}. The
integer points on $\x.\c=0$ form a 3-dimensional lattice
$\mathsf{\Lambda}_{\c}\subset\ZZ^4$ say, whose
determinant is $\|\c\|_2=\sqrt{c_1^2+\dots+c_4^2}$.  We choose
$\mathbf{e}^{(1)},\mathbf{e}^{(2)},\mathbf{e}^{(3)}$ as a basis for
the lattice and set
\[q(\y)=Q_{\c}(\y)=
Q(y_1\mathbf{e}^{(1)}+y_2\mathbf{e}^{(2)}+y_3\mathbf{e}^{(3)}).\]
If we suppose that $Q$ has underlying symmetric matrix $\mathbf{M}$,
then $Q_\c$ clearly has underlying $3\times 3$ matrix 
\begin{equation}\label{five}
\mathbf{M}_\mathbf{c}= \mathbf{E}^t \mathbf{M} \mathbf{E},
\end{equation}
where
$\mathbf{E}$ is  the $4 \times 3$ matrix with columns 
$\mathbf{e}^{(1)},\mathbf{e}^{(2)},\mathbf{e}^{(3)}$.
The following result is a generalisation of \cite[Eq.~(20)]{qpub},
which only deals  
with diagonal  forms $Q$.

\begin{lemma}\label{lem:disc}
We have $\det \mathbf{M}_\c=Q^*(\c)$, where 
$Q^*$ is the dual form.
\end{lemma}

\begin{proof}
We let $\mathbf{E}_i$ denote the square  matrix obtained by deleting the $i$th
row from $\mathbf{E}$, for $1\leq i\leq 4$. Put 
$$
\mathbf{d}=(-\det \mathbf{E}_1, \det \mathbf{E}_2, -\det
\mathbf{E}_3,\det \mathbf{E}_4) 
$$
and let   $i\in \{1,2,3,4\}$.  Since the $4\times 4$ matrix with
columns $\mathbf{e}^{(i)},\mathbf{e}^{(1)},\dots, \mathbf{e}^{(3)}$
has determinant $0$,  
 it follows that  $\mathbf{d}.\mathbf{e}^{(i)}=0$.  But this implies that 
$\mathbf{d}$ belongs to the dual of $\mathsf{\Lambda}$, in the notation 
of Lemma \ref{lattice}, which is equal to 
$\langle \c \rangle_\ZZ$.  
Now   $\ma d$ is clearly non-zero, since   $\rank \ma E=3$.
Moreover, $\ma d$  is primitive since 
 it would otherwise follow that there is a prime $p$ for which 
the vectors  $\mathbf{e}^{(i)}$ are
linearly dependent modulo $p$,   contradicting the fact that they
extend to a basis of $\ZZ^4$.
Hence we have shown that $\ma d =\pm \c$.

To calculate $\det \ma M_\c$ we invoke the Cauchy--Binet formula.
It follows from \eqref{five} that 
\begin{align*}
\det \M_\c 
&=\sum_{i=1}^4 \det(\ma E_i^t) \det (\ma M_i \ma E)
=\sum_{i,j=1}^4 \det(\ma E_i)\det (\ma M_{i,j}) \det (\ma E_j),
\end{align*}
where $\ma M_i$ is the $3\times 4$ matrix obtained by deleting the
$i$th row from $\ma M$ and $\ma M_{i,j}$ is the square matrix
obtained by further deleting the $j$th column. The lemma now follows
on observing that  
$\det (\ma M_{i,j})=(-1)^{i+j}(\ma M^{\text{adj}})_{i,j} $ and recalling that 
$\ma d=\pm \c$.
\end{proof}
\medskip

To apply Lemma \ref{two} we will also need to understand $D(q)^\square$
and $\chi_q$ for $q=Q_{\c}$. If $\ee^{(1)}, \ee^{(2)}, \ee^{(3)}$ are
a basis for $\mathsf{\Lambda}_{\c}$, as before, we may extend to a basis
of $\ZZ^4$ by adding $\ee^{(4)}$, say. There are therefore
integers $a,b,c,d$ such that 
\begin{equation}\label{eq:shape}
Q(z_1 \ee^{(1)}+ \dots +z_4 \ee^{(4)})=Q_\c(z_1,z_2,z_3)+z_4(az_1 +
bz_2+cz_3+dz_4).
\end{equation}
The left hand side is a quaternary quadratic form of
discriminant 
$\Delta_Q$, since the $4\times 4$ matrix with columns 
$\ee^{(1)}, \dots , \ee^{(4)}$ has determinant $\pm 1$.
For any odd prime $p$ and any positive integer $\xi$
we may apply a unimodular transformation 
to the variables $z_1,z_2,z_3$ in order to 
diagonalize $Q_\c$ modulo $p^\xi$. In this way, 
we may assume  that 
$Q_\c$ has underlying matrix $\Diag(A,B,C)$, with $v_p(A)\le
v_p(B)\le v_p(C)$. 
In particular, if $p\mid Q^*(\c)$ then $p\mid\det(Q_{\c})$ 
and hence $p\mid C$. 
It follows from \eqref{eq:shape} that 
$$
4\Delta_Q\equiv -a^2BC-b^2AC-c^2AB+4dABC \bmod{p^\xi}.
$$
Thus if $p\nmid \Delta_Q$ and $p \mid Q^*(\c)$ then
$$
\chi_{Q_\c}(p)=
\left(\frac{-AB}{p}\right)=
\left(\frac{\Delta_Q}{p}\right).
$$
Next, if 
$p^v\| D(Q_{\c}) 
$ then, taking $\xi=v$, we see that
$p^v\mid\Delta_Q$.  
When $p=2$, one may diagonalize $4Q_\c$  using an integer matrix of
determinant $2$. 
Arguing as above one then finds that if $2^v\| D(Q_{\c})$ then
$2^v\mid 2^8\Delta_Q$. 
Once combined with our treatment of the odd primes, 
this yields
$D(Q_\c) \mid 2^8\Delta_Q$. On the other hand, it is clear that
$D(q)^3\mid \det(\mathbf{A}^{{\rm adj}})$, whence
$D(q)^3\mid \Delta_q^2$. It follows that we also have
$D(Q_\c)^3\mid Q^*(\c)^2$.  Thus $D(Q_\c)^3$ divides
$2^{24}(\Delta_Q^3,Q^*(\c)^2)$, so that
\[D(Q_\c)^\square \ll  (\Dbad^3,Q^*(\c)^2)^{1/3}. \]
It therefore follows from Lemma
\ref{lem:disc} that the lattices in Lemma \ref{two} satisfy
\beql{LLB}
\det(\mathsf{\Lambda}_i)\gg \frac{|Q^*(\c)|}{(\Dbad^3,Q^*(\c)^2)^{1/2}}.
\eeq
when $q=Q_{\c}$.

According to Lemma \ref{two}, if $Q_{\c}(\y)=0$ then $\y$ must belong
to one of the lattices $\mathsf{\Lambda}_i$.
We write
\[\widehat{\mathsf{\Lambda}}_i=\{y_1\mathbf{e}^{(1)}+y_2\mathbf{e}^{(2)}
+y_3\mathbf{e}^{(3)}:\, \y\in\mathsf{\Lambda}_i\},\]
where $\mathbf{e}^{(1)},\mathbf{e}^{(2)},\mathbf{e}^{(3)}$ are a basis for
$\mathsf{\Lambda}_{\c}$ as before.
Thus $\widehat{\mathsf{\Lambda}}_i$ is a 3-dimensional lattice in
$\ZZ^4$.  Moreover, if $\x$ is an integer solution of
$Q(\x)=\c.\x=0$,
then $\x\in\widehat{\mathsf{\Lambda}}_i$ for some index $i$.
We proceed to compute the determinants of these lattices.
\begin{lemma}
  We have
  \[\det(\widehat{\mathsf{\Lambda}}_i)=\|\c\|_2\det(\mathsf{\Lambda}_i)
  \gg \frac{\|\c\|_2.|Q^*(\c)|}{(\Dbad^3,Q^*(\c)^2)^{1/2}}.\]
\end{lemma}
  
\begin{proof}
If $\mathsf{\Lambda}_i\subset\ZZ^3$ has a basis $\mathbf{h}^{(1)},
\mathbf{h}^{(2)}, \mathbf{h}^{(3)}$, then
\[\det(\mathsf{\Lambda}_i)=|\det(\mathbf{H})|,\]
where $\mathbf{H}$ is the $3\times 3$ matrix with columns
$\mathbf{h}^{(1)}, \mathbf{h}^{(2)}, \mathbf{h}^{(3)}$. Moreover if 
$\mathbf{E}$ is the $4\times 3$ matrix with columns
$\mathbf{e}^{(1)}, \mathbf{e}^{(2)}, \mathbf{e}^{(3)}$, then 
$\widehat{\mathsf{\Lambda}}_i$ will have a basis consisting of the
columns of $\mathbf{E}\mathbf{H}$. It then follows that
\[\left(\det(\widehat{\mathsf{\Lambda}}_i)\right)^2=
\det(\mathbf{H}^T\mathbf{E}^T\mathbf{E}\mathbf{H}).\]
Since $\mathbf{H}$ and $\mathbf{E}^T\mathbf{E}$ are both $3\times 3$
matrices, and
\[\det(\mathbf{E}^T\mathbf{E})=\left(\det(\mathsf{\Lambda}_{\c})\right)^2
=\|\c\|_2^2,\]
we deduce that
\[\left(\det(\widehat{\mathsf{\Lambda}}_i)\right)^2=\|\c\|_2^2\det(\mathbf{H})^2
=\|\c\|_2^2\left(\det(\mathsf{\Lambda}_i)\right)^2.\]
Thus
$\det(\widehat{\mathsf{\Lambda}}_i)=\|\c\|_2\det(\mathsf{\Lambda}_i).$
The result now follows via (\ref{LLB}).
\end{proof}

We now have to consider primitive integer vectors $\x$ which lie in one of
the lattices $\widehat{\mathsf{\Lambda}}_i$, as well as being in one
of the ellipsoids $E_j$ of Lemma \ref{Q*B}.  We can therefore
use Lemma \ref{lattice} with $V=\{\x\in\RR^4:\x.\c=0\}$
to deduce that, for each index $i$, and each
ellipsoid $E_j$, the relevant values of $\x$ take the form
$\sum_k\lambda_k\mathbf{f}^{(k)}$, with $|\lambda_k|\le L_k$, and
\begin{align*}
L_1L_2L_3&\ll \frac{|Q^*(\c)|.\|Q\|^3 B^3}{|\c|^2|\Delta_Q|^{3/2}}
\frac{(\Dbad^3,Q^*(\c)^2)^{1/2}}{|\c|.|Q^*(\c)|}\nonumber\\
&=
\left(\frac{\|Q\|B{(\Dbad^3,Q^*(\c)^2)^{1/6}}}{|\c|.|\Delta_Q|^{1/2}}\right)^3.
\end{align*}
We remark at this point that one can alternatively give a bound
\[\ll \frac{B^3\Dbad^{3/2}}{|\c|.|Q^*(\c)|},\]
which can be superior in certain circumstances.  However the factor
$Q^*(\c)$ in the denominator is rather inconvenient.

We now apply Lemma \ref{V3} to show that there are
\[\ll
1+\frac{\|Q\|B(\Dbad^3,Q^*(\c)^2)^{1/6}}{|\c|.|\Delta_Q|^{1/2}}\]
primitive solutions, for each lattice $\widehat{\mathsf{\Lambda}}_i$ and
each ellipsoid $E_j$. It transpires that the highest common factor
term is in a rather awkward shape, because it involves the square of
$Q^*(\c)$.  We shall replace it with a weaker upper bound, which is chosen
in such a way that  it will eventually cancel with extra factors that
come into play in the next section. First note that  if $m$ and $n$ are
non-zero integers, and $h=(m,n)$, then
\[(m,n^2)^{1/6}\le\frac{m^{1/12}h}{(m,h^4)^{1/4}}.\]
This is easily proved, by considering the case in which $m$ and $n$
are powers of a single prime. Taking $m=\Dbad^3$ and $n=Q^*(\c)$ we
deduce that
\[(\Dbad^3,Q^*(\c)^2)^{1/6}\le\Dbad^{1/4}\frac{h}{(\Dbad^3,h^4)^{1/4}},\]
with $h=(\Dbad^3,Q^*(\c))$.
We therefore have the following conclusion.

\begin{lemma}\label{NewB}
Let
\begin{equation}\label{eq:R}
R(N)=\prod_{\substack{p^\xi \| N\\ \text{$p \mid  2 \Delta_Q$}}}
\tau(p^\xi )\prod_{\substack{p^\xi  \| N\\  p\nmid 2\Delta_Q}}
\left\{\sum_{k=0}^\xi \left(\frac{\Delta_Q}{p^k}\right)\right\}.
\end{equation}
Then if $Q_{\c}$ is non-singular there is an integer
$h\mid (\Dbad^3,Q^*(\c))$ such that there are
  \[\ll R(Q^*(\c))
  \left(1+
  \frac{\|Q\|B\Dbad^{1/4}h}{|\c|.|\Delta_Q|^{1/2}(\Dbad^3,h^4)^{1/4}}\right)
  \log\left(2+\frac{|\c|^2\|Q\|^3}{|Q^*(\c)|}\right)\]
  primitive vectors $\x$ with $|\x|\le B$, for which $Q(\x)=\c.\x=0$.
\end{lemma}
\medskip

Assume for the time being that $\Delta_Q\neq \square$.
Returning to  \eqref{eq:step1}, we recall that
$$
N(B)\leq \sum_{\substack{\c\in \ZZp^4\\ |\c|\ll B^{1/3}}} 
\#\left\{ \x\in\ZZp^4:\,\x.\c=0,\,Q(\x)=0,\,|\x|\le B\right\}.
$$
As is well-known the rank of a quadratic form drops by at most $2$ on
any hyperplane. Thus $\rank Q_\mathbf{c}\geq 
2$. If $\rank Q_\mathbf{c}=2$ then the conic 
$Q_\mathbf{c}=0$ is a union of two lines. 
However the assumption that $\Delta_Q\neq
\square$ implies that there are no $\QQ$-lines contained in the
quadric surface $Q=0$. Thus if $\rank Q_\mathbf{c}=2$ then the conic
$Q_{\c}=0$ has exactly one
rational point, so that  the overall contribution from this case is
\[\le\#\left\{ \c\in\ZZp^4:\,|\c|\ll B^{1/3},\,Q^*(\c)=0\right\}.\]
However $Q^*$ is nonsingular, so that the number of such $\c$ is
$O(B)$, by  Heath-Brown \cite[Theorem 1]{annal}, for example. 
It  now follows from Lemma \ref{NewB} that
\beql{eq:NB}
N(B)\ll B+S+B\frac{\|Q\|\Dbad^{1/4}}{|\Delta_Q|^{1/2}}
\max_h\frac{h}{(\Dbad^3,h^4)^{1/4}}S_h,
\eeq
the maximum being for  
$h\mid  \Dbad^3$, where we have written 
\[S=\sum_{\substack{|\c|\ll B^{1/3},\,Q^*(\c)\not=0\\
    \c\in \ZZp^4}}
R(Q^*(\c))\log\left(2+\frac{|\c|^2\|Q\|^3}{|Q^*(\c)|}\right)\]
and
\[S_h=\sum_{\substack{|\c|\ll B^{1/3},\,Q^*(\c)\not=0\\
    \c\in \ZZp^4,~ h\mid Q^*(\c)}}
\frac{R(Q^*(\c))}{|\c|}\log\left(2+\frac{|\c|^2\|Q\|^3}{|Q^*(\c)|}\right).\]

It would be relatively straightforward to estimate these sums trivially, if
we permit ourselves the use of  the standard divisor sum bound
$R(N)\ll N^\ve$.   However, we shall need to show that 
$R(Q^*(\c))$ has order $1$ on average, ignoring possible factors of $\Dbad$.
Furthermore, in order to cope with the term
$|Q^*(\c)|$ in the logarithm, we shall need to study the average of
$R(Q^*(\c))$ in short intervals.

\section{Multiplicative functions over values of a quadratic form}

In this section we show how to handle averages of $R(Q^*(\c))$.
We begin by studying the function
\[\rho(m)=\#\{\x\in (\ZZ/m\ZZ)^4: Q^*(\x)\equiv 0\bmod{m}\},\]
which is clearly multiplicative. 
The properties of $\rho(p^k)$ that we require are summarized as follows.

\begin{lemma}\label{lem:rho}
  We have
 \[\rho(p)=p^3+\left(\frac{\Delta_Q}{p}\right)(p^2-p)\]
 when $p\nmid 2\Dbad$.  Moreover
 $\rho(p^k)\le 4kp^{3k}(\Dbad^3, p^{4k})^{1/4}$ for all $k\ge 1$
 and all primes $p$.
\end{lemma}

\begin{proof}
We start from the relation
\[\rho(p^k)=p^{-k}\sum_{a=1}^{p^k}\sum_{\x\modd{p^k}}S(a;p^k),\]
where
\[S(a;p^k)=\sum_{\x\modd{p^k}}e_{p^k}(aQ^*(\x)).\]
When $p^f\| a$ with $f\le k$ we have
\[S(a;p^k)=p^{4f}\sum_{\x\modd{p^{k-f}}}e_{p^{k-f}}(ap^{-f}Q^* (\x)),\]
so that
\beql{rhbas}
\rho(p^k)=p^{-k}\sum_{f=0}^k p^{4f}
\sum_{\substack{b=1\\ (b,\,p^{k-f})=1}}^{p^{k-f}} S(b;p^{k-f})
=p^{3k}\sum_{g=0}^k p^{-4g}
\sum_{\substack{b=1\\ (b,\,p^g)=1}}^{p^g} S(b;p^g) . 
  \eeq

To prove the first assertion of the lemma we take $k=1$ and begin by examining
$p\nmid 2\Delta_Q$. We may then  diagonalize $Q^*$ modulo $p$ as
${\rm Diag}(d_1,\ldots,d_4)$ say, with
\[d_1\ldots d_4\equiv\det(Q^*)\equiv\Delta_Q^3\modd{p}.\]
It follows that
\[S(b;p)=\prod_{i=1}^4G(bd_i,p),\]
where
\[G(b,p)=\sum_{x=1}^pe_p(bx^2)=\ve_p\left(\frac{b}{p}\right)\sqrt{p}\]
is a Gauss sum, with $\ve_p=1$ for $p\equiv 1\modd{4}$ and $\ve_p=i$ for
$p\equiv 3\modd{4}$. We then find that
\[S(b;p)=\left(\frac{\Delta_Q}{p}\right)p^2\]
and the first assertion of Lemma \ref{lem:rho} follows in the case
$p\nmid 2\Delta_Q$. When $p\|\Delta_Q$ for an odd prime $p$ we see
that $Q^*$ has rank 1 modulo $p$, and thence that $\rho(p)=p^3$.

For the second assertion of the lemma we note that
the terms $g=0$ and 1 in (\ref{rhbas})
produce 
$p^{3k-3}\rho(p)$.   This is at most $2p^{3k}$ when
$p$ does not divide the matrix $\mathbf{M}^{{\rm adj}}$ of $Q^*$, and
is  $p^{3k+1}$ otherwise. If $p$ does divide $\mathbf{M}^{{\rm adj}}$
we will have $p^4\mid\Delta_Q^3$, so that
$p^{3k-3}\rho(p)\le 2p^{3k}(\Dbad^3,p^{4k})^{1/4}$ in every case. 

When $g\ge 2$ we use Cauchy's inequality to deduce that
\begin{align*}
  |S(b;p^g)|^2&\le\sum_{\x\modd{p^g}}\sum_{\mathbf{y}\modd{p^g}}
  e_{p^g}(bQ^*(\mathbf{y})-bQ^*(\x))\\
  &=\sum_{\x\modd{p^g}}\sum_{\mathbf{z}\modd{p^g}}
  e_{p^g}(bQ^*(\mathbf{z}+\x)-bQ^*(\x))\\
  &\le\sum_{\mathbf{z}\modd{p^g}}\left|\sum_{\x\modd{p^g}}
  e_{p^g}(2b\mathbf{z}^T\mathbf{M}^{\mathrm{adj}}\x)\right|. 
\end{align*}
We can put 
$\mathbf{M}^{\mathrm{adj}}$  
into Smith Normal Form, by writing
$\mathbf{M}^{\mathrm{adj}}
=\mathbf{A}^T\mathbf{D}\mathbf{B}$ where $\mathbf{A}$ and
$\mathbf{B}$ are unimodular integer matrices and 
$\mathbf{D}={\rm Diag}(D_1,\ldots,D_4)$ is a
diagonal matrix  with
$D_1\ldots D_4=\det(Q^*)=\Delta_Q^3$. Then
\begin{align*}
  |S(b;p^g)|^2&\le \sum_{\mathbf{z}\modd{p^g}}\left|\sum_{\x\modd{p^g}}
  e_{p^g}(2b\mathbf{z}^T\mathbf{D}\x)\right|\\
  &=p^{4g}\card\{\mathbf{z}\modd{p^g}:
  2b\mathbf{D}\mathbf{z}\equiv\mathbf{0}\modd{p^g}\}.
  \end{align*}
Since $p\nmid b$ there are $(2D,p^g)$ solutions to
$2bDz\equiv 0\modd{p^g}$, whence
\[|S(b;p^g)|^2\le p^{4g}\prod_{i=1}^4(2D_i,p^g)\le
16p^{4g}\left( D_1\dots D_4,p^{4g}\right).\]
It follows that
\[|S(b;p^g)|\le 4p^{2g}(\Delta_Q^3,p^{4g})^{1/2}\]
for $g\ge 2$.
When $p\nmid\Dbad$ we have $(\Delta_Q^3,p^{4g})\le p^3$. In this case
the terms of (\ref{rhbas}) with $2\le g\le k$ contribute at most
\[p^{3k}\sum_{g=2}^k p^{-4g}
\sum_{\substack{b=1\\ (b,p^g)=1}}^{p^g}|S(b;p^g)|\le  
4p^{3k}\sum_{g=2}^{\infty} p^{-g+3/2}(1-p^{-1})\le 3p^{3k}.\]
The terms $g=0$ and $g=1$ combine to produce $p^{3k-3}\rho(p)\le
2p^{3k}(\Dbad,p^k)$, whence $\rho(p^k)\le 5p^{3k}$ for
$p\nmid\Dbad$. This is satisfactory for the lemma.

Similarly when $p\mid\Dbad$ we observe that
  \[(\Delta_Q^3,p^{4g})^{1/2}\le (\Delta_Q^3,p^{4g})^{1/4}p^g,\]
so that terms with $2\le g\le k$ contribute at most
 \begin{align*}
  p^{3k}\sum_{g=2}^k p^{-4g}
\sum_{\substack{b=1\\ (b,p^g)=1}}^{p^g}|S(b;p^g)|&\le
4p^{3k}\sum_{g=2}^k (\Delta_Q^3,p^{4g})^{1/4}\\ 
&\le  4(k-1)p^{3k}(\Dbad^3,p^{4k})^{1/4}.
\end{align*}
Adding in the terms for $g=0$ and $g=1$, as before, we therefore find that
$\rho(p^k)\le 4kp^{3k}(\Dbad^3,p^{4k})^{1/4}$.
 The second part of the lemma then follows.
\end{proof}

We can now describe the average of $R(Q^*(\c))$ which we plan to estimate.
Given any $\u\in \RR^4$, write
\[\mathcal{R}=\{\x\in \RR^4:  |\x-\u|\leq X,\, Q^*(\x)\not=0\}.\]
This set has measure $O(X^4)$. 
We are interested here in the size of the sum 
\[S^{(h)}(X)=\sup_{\u\in \RR^4}\;
\sum_{\substack{ \x\in \ZZ^4\cap \mathcal{R}\\ h\mid Q^*(\x)}} R(|Q^*(\x)|).\]
By developing a variant of  familiar arguments of Shiu \cite{shiu}, we shall
establish the  following estimate. 

\begin{theorem}\label{t:nair}
Suppose that
  \begin{equation}\label{eq:assume}
\sup_{\x\in \cR} |Q^*(\x)|\leq X^A,
  \end{equation}
  for some constant $A$, and let $\ve>0$ be given.  Then if
  $h\mid\Dbad^3$ we have 
  \[S^{(h)}(X)\ll_{A,\ve} \Dbad^{\ve}h^{-1}(\Dbad^3,h^4)^{1/4}
  \mathfrak{S}\frac{X^4}{\log X},\]
uniformly for $h\le X^{1-\ve}$, where 
  \[\mathfrak{S}=\prod_{p\le X}\left(1+\frac{R(p)}{p}\right).\]
\end{theorem}

For our argument we will use a parameter $z=X^{\eta}$ with
$\eta>0$. We will eventually choose $\eta=\ve/13$.  However the structure
of the proof will be clearer if we leave $\eta$ undetermined for the
time being.  In the course of the proof we will allow all the
constants implied by the $O(\ldots)$, $\ll$ and $\gg$ notations to
depend on $A$, $\ve$ and $\eta$.

An inspection of \eqref{eq:R} shows that $R(p^{e+f})\le R(p^e)R(p^f)$
except possibly when $p\nmid 2\Delta_Q$ with $e$ and $f$ both odd.
Thus
\beql{Ri}
R(uv)\le R(u)R(v)\le\tau(u)R(v)
\eeq
unless there is some prime $p\nmid 2\Delta_Q$ which divides both $u$
and $v$ to an odd power.
For any $\x\in\ZZ^4\cap\cR$ with $h\mid Q^*(\x)$ we now let
$|Q^*(\x)|=hp_1p_2\ldots p_r$ with $p_1\le p_2 \le 
\ldots \le p_r$, and
choose $j\in [0,r]$ maximally such that $a=p_1\ldots p_j\le z^2$. We
then set $b=p_{j+1}\ldots p_r$.  We will consider four
cases. If $a\le z$ then since $j$ was chosen maximally we must have
$j=r$ or $p_{j+1}> z\ge a$. In both of these situations (\ref{Ri})
shows that $R(|Q^*(\x)|)\le\tau(hb)R(a)$. Moreover, since
$p_{j+1}\ge z$ we have 
\[z^{r-j}\le p_{j+1}^{r-j}\le p_{j+1} p_{j+2}\ldots p_r\le|Q^*(\x)|\le X^A,\]
so that $r-j\le A(\log X)/(\log z)=A/\eta$. Thus $\tau(b)\ll 1$ and
\[R(|Q^*(\x)|)\leq\tau(hb)R(a)\le\tau(h)\tau(b)R(a)\ll h^{\eta}R(a)\]
when $a\le z$. We remind the reader that in this case we have
$P^{-}(b)>z$, where $P^{-}(n)$ is the smallest prime factor of $n$
(and $P^{-}(1)=\infty$).  Similarly we write $P^{+}(n)$ for the
largest prime factor of $n$, with $P^+(1)=1$.

The next case to examine is that in which $z<a\le z^2$ and
$p_{j+1}>p_j>\log X$. Here again we find from (\ref{Ri}) that
$R(|Q^*(\x)|)\le\tau(hb)R(a)$. This time we note that
\[p_j^{r-j}\le p_{j+1} p_{j+2}\ldots p_r\le|Q^*(\x)|\le X^A,\]
whence $r-j\le A(\log X)/(\log p_j)$ and
\[\tau(b)\le 2^{\Omega(b)}=2^{r-j}\le X^{A/\log p_j}.\]
Proceeding as before we are led to the bound
\beql{R2}
R(|Q^*(\x)|)\ll h^{\eta}X^{A/\log p_j}R(a),
\eeq
in which we have $P^+(a)=p_j<p_{j+1}=P^-(b)$.

When $z<a\le z^2$ with $p_{j+1}=p_j>\log X$ we are unable to use
(\ref{Ri}) in quite the same way. In view of the construction of $a$
and $b$ the only prime factor which they can share is $p_j$.  If $p_j$
divides one or both of $a$ or $b$ to an even power we may derive
(\ref{R2}) as before. So we now suppose that $p_j$ divides each of $a$
and $b$ to an odd power.  In this situation we set
$a'=ap_{j+1}$ and $b'=b/p_{j+1}$ so that
\[R(|Q^*(\x)|)\ll h^{\eta}X^{A/\log p_j}R(a'),\]
by the argument leading to (\ref{R2}). Since $p_{j+1}=p_j\le a\le z^2$ we
then have $z<a'\le z^4$ and $P^+(a')=p_j=p_{j+1}\le P^-(b')$.

The remaining case is that in which $z<a\le z^2$ but
$p_j\le\log X$, and here we merely use the fact that
\[R(|Q^*(\x)|)\ll X^{\eta}.\]

In the third case we change notation writing $a$ in place of $a'$. 
We then see that 
\[S^{(h)}(X)\ll h^{\eta}\left\{T_1(X)+T_2(X)+X^{\eta}T_3(X)\right\},\]
with
\begin{align*} 
T_1(X)&=\sum_{a\le z}R(a)U(ah;z),\\
T_2(X)&=\sum_{\log X<p_j\le z^2}X^{A/\log p_j}   
\twosum{z< a\le z^4}{P^+(a)=p_j}R(a)U(ah;p_j), 
\end{align*}
and
\[T_3(X)=\twosum{z< a\le z^2}{P^+(a)\le\log X}U(ah;2),\] 
where we have defined
\[U(a;\tau)=\#\left\{\x\in \ZZ^4\cap \mathcal{R}:
a\mid Q^*(\x), P^-(Q^*(\x)/a)\ge\tau\right\}.\]
This is estimated in the following lemma, 
in which $\varpi$ is defined in \eqref{eq:phid} and which we shall prove later. 
\begin{lemma}\label{sieve}
If  $a\le Xz^{-11}$ we have
\[U(a;\tau) \ll\varpi(\Dbad)\varpi(a)\frac{X^4\rho(a)}{a^4\log\tau},\]
for $2\le\tau\le z^2$.
\end{lemma}

Taking this for granted for the time being, we need to consider
$\rho(ah)$. We define a
multiplicative function $\rho_0$ by setting
$$
\rho_0(p^e)=
\begin{cases}
4ep^{3e}(\Dbad^3,p^{4e})^{1/4}, & \text{ if $p\mid\Dbad$,}\\ 
\rho(p^e), & \text{ if  $p\nmid\Dbad$},
\end{cases}
$$
for any $e\geq 1$.
 Then if $a=a_1a_2$ with $a_1\mid\Dbad^{\infty}$ and
$(a_2,\Dbad)=1$, we will have
\[\rho(ah)=\rho(a_1h)\rho(a_2)\le\rho_0(a_1)\rho_0(h)\rho_0(a_2)
=\rho_0(h)\rho_0(a).\]
In particular we now see that
$\varpi(ah)\rho(ah)\ll h^{3+\eta}(\Dbad^3,h^4)^{1/4}\varpi(a)\rho_0(a)$.

Thus if $h\le Xz^{-13}$ we have
\beql{st}
S^{(h)}(X)\ll X^4\varpi(\Dbad)h^{-1+2\eta}(\Dbad^3,h^4)^{1/4}\Sigma,
\eeq
where
\[\Sigma=
\left\{\frac{\Sigma_1}{\log X}
+\sum_{\log X<p_j\le z^2}\frac{X^{A/\log p_j}}{\log p_j}
\Sigma_2(p_j)+X^{\eta}\Sigma_3\right\},\]
with
\[\Sigma_1=\sum_{a\le z}\frac{\rho_0(a)\varpi(a)R(a)}{a^4},\]
\[\Sigma_2(p_j)=\twosum{z<a\le z^4}{P^+(a)=p_j}  
\frac{\rho_0(a)\varpi(a)R(a)}{a^4},\]
and
\[\Sigma_3=
\twosum{z< a\le z^2}{P^+(a)\le\log X}\frac{\rho_0(a)\varpi(a)}{a^4}.\]
Note that the condition on $h$ is just $h\le X^{1-\ve}$, 
in view of the choice $\eta=\ve/13$.

We begin our analysis of these sums by examining $\Sigma_2(p_j)$.
Since $p_j$ tends to infinity with $X$ we may put
\[\delta=\delta(p_j)=
\frac{A+1}{\eta\log p_j}\in (0,\min\{\tfrac18, \tfrac{\eta}{2}\}]\] 
for  large enough $X$, so that 
\[\Sigma_2(p_j)\le \twosum{z <a\le z^4}{P^+(a)=p_j}  
\frac{\rho_0(a)\varpi(a)R(a)}{a^4}\left(\frac{a}{z}\right)^{\delta}.\]
Recalling that $z=X^{\eta}$ we then have
\[z^{-\delta}X^{A/\log p_j}= X^{-1/\log p_j}.\]
Moreover
\[\twosum{z <a\le z^4}{P^+(a)=p_j}\frac{\rho_0(a)\varpi(a)R(a)}{a^{4-\delta}}
\le\sum_{\substack{a=1\\ P^+(a)=p_j}}^{\infty}
\frac{\rho_0(a)\varpi(a)R(a)}{a^{4-\delta}},\]
which factorizes as
\beql{sp}
\prod_{p\le p_j}\sigma_p,
\eeq
say.  We therefore have
\beql{stop}
\sum_{\log X<p_j\le z^2}\frac{X^{A/\log p_j}}{\log p_j}
\Sigma_2(p_j)
\leq \sum_{\log X<p_j\le z^2}\frac{X^{-1/\log p_j}}{\log p_j}
\prod_{p\le p_j}\sigma_p.
\eeq

We shall prove the following estimates.
\begin{lemma}\label{lem:sig}
  When $p<p_j$ does not divide $2\Dbad$ we have
  \[\sigma_p=1+\frac{R(p)}{p}+O(p^{-3/2})
  +O\left(\frac{\log p}{p\log p_j}\right).\]
When $p^v\|2\Dbad$ for $v\ge 1$ we have
  \beql{C?}
    \sigma_p\ll (v+1)^3p^{v\delta}. 
  \eeq
Finally, if $p=p_j$ does not divide $2\Dbad$ we have
\[\sigma_p\ll p^{-1}.\]  
  \end{lemma}

\begin{proof}
For primes $p\not=p_j$ with $p\nmid 2\Dbad$ we have
\[\sigma_p=
1+\sum_{e=1}^\infty\frac{\rho_0(p^e)\varpi(p^e)R(p^e)}{p^{(4-\delta)e}}.\]
Moreover $\rho_0(p^e)\le 4ep^{3e}$, by Lemma \ref{lem:rho}.
Since $R(p^e)\le e+1$ and $\delta\le 1/4$ we find that
\[\sum_{e=2}^\infty\frac{\rho_0(p^e)\varpi(p^e)R(p^e)}{p^{(4-\delta)e}}\ll
\sum_{e=2}^\infty\frac{(e+1)e}{p^{3e/4}}\ll p^{-3/2}.\]
For $e=1$ we see via Lemma \ref{lem:rho} that
$\rho_0(p)=p^3+O(p^2)$ and hence that
\[\frac{\rho_0(p)\varpi(p)R(p)}{p^4}p^{\delta}=\frac{R(p)}{p}
+O\left(\frac{\log p}{p\log p_j}\right)+O(p^{-3/2}).\]
The first assertion of the lemma then follows.

Similarly, when $p\mid 2\Dbad$ we find that
\[\sigma_p\le 1+\sum_{e=1}^{\infty}
\frac{4e(e+1)}{p^{(1-\delta)e}}\varpi(p)(\Dbad,p^{e}).\] 
We can estimate this sum by
breaking it at $e=v$, where $p^v\|\Dbad$.  One then finds that
\[\sigma_p\ll (v+1)^3p^{v\delta},\]
as required.

In the case $p=p_j$ we have
\[\sigma_p=\sum_{e=1}^\infty\frac{\rho_0(p^e)\varpi(p^e)R(p^e)}{p^{(4-\delta)e}}.\]
The analysis is now just as above, except that there is no term
corresponding to $e=0$. This completes the proof of the lemma. 
\end{proof}

We can now use Lemma \ref{lem:sig} to estimate the product
(\ref{sp}). The principle we employ is that if $a_n\ge 0$ and
$\sum_{n=1}^N|b_n|=B$, 
then
\[\prod_{n=1}^N(1+a_n+b_n)\le e^B\prod_{n=1}^N(1+a_n).\]
Thus
\[\prod_{p<p_j,\,p\nmid 2\Dbad}\sigma_p\ll \prod_{p<p_j,\,p\nmid
  2\Dbad}\left(1+\frac{R(p)}{p}\right).\]
On the other hand, if the implied constant in (\ref{C?}) is $C_0$, then
\[\prod_{p\le p_j,\,p\mid 2\Dbad}\sigma_p\ll
\tau(\Dbad)^{C_0+3}\Dbad^{\delta}\ll\Dbad^{\eta},\]
since $\delta\leq \eta/2$.  Using  the final part
of Lemma \ref{lem:sig} to bound $\sigma_{p_j}$ when $p_j\nmid 2\Dbad$
we therefore have
\[  \prod_{p\le p_j}\sigma_p\ll \Dbad^\eta ~p_j^{-1}\prod_{p< p_j}
 \left(1+\frac{R(p)}{p}\right)\ll\Dbad^\eta~p_j^{-1}\mathfrak{S}\]
when $p_j\nmid 2\Dbad$, and
\[\prod_{p\le p_j}\sigma_p\ll \Dbad^\eta \mathfrak{S}\]
if $p_j\mid 2\Dbad$.

It then follows from (\ref{stop}) that
\begin{align*}
\sum_{\log X<p_j\le z^2}\frac{X^{A/\log p_j}}{\log p_j}
    \Sigma_2(p_j)
  \ll \Dbad^{\eta}\mathfrak{S}  
\left\{\sum_{p_j\le z^2}\frac{X^{-1/\log p_j}}{p_j\log p_j}+
\sum_{p_j\mid 2\Dbad}\frac{X^{-1/\log p_j}}{\log p_j}\right\}.
\end{align*}
However, 
\begin{align*}
\sum_{X^{1/(r+1)}<p\le X^{1/r}}X^{-1/\log p}(p\log p)^{-1}&\le
e^{-r}\sum_{p>X^{1/(r+1)}}(p\log p)^{-1}\\
&\ll  e^{-r}r(\log X)^{-1},
\end{align*}
uniformly for all positive integers $r$, whence
\[\sum_{p_j\le z^2}\frac{X^{-1/\log p_j}}{p_j\log p_j}\ll (\log
X)^{-1}.\]
Moreover $X^{-1/\log p_j}/\log p_j\ll(\log X)^{-1}$ for any prime
$p_j$, so that
\[\sum_{p_j\mid 2\Dbad}\frac{X^{-1/\log p_j}}{\log p_j}\ll
\tau(\Dbad)(\log X)^{-1}.\]
According to (\ref{st}) the terms involving $\Sigma_2(p_j)$ therefore
make a contribution
\begin{align*}
&\ll \frac{X^4}{\log X}\varpi(\Dbad)\Dbad^{\eta}\tau(\Dbad)h^{-1+2\eta}
(\Dbad^3,h^4)^{1/4}\mathfrak{S}\\
&\ll \frac{X^4}{\log X}\Dbad^{2\eta}h^{-1+2\eta}
(\Dbad^3,h^4)^{1/4}\mathfrak{S},
\end{align*}
which is satisfactory for Theorem \ref{t:nair}, provided that we
choose $8\eta<\ve$, since $h\leq \Dbad^3$.  
Indeed, we mentioned earlier that the appropriate
choice is $\eta=\ve/13$. 

The treatment of $\Sigma_1$ is now straightforward.
We have
\[\Sigma_1=\sum_{a\le z^2}\frac{\rho_0(a)\varpi(a)R(a)}{a^4}
\le \sum_{\substack{a=1 \\P^+(a)\le z^2}}^\infty\frac{\rho_0(a)\varpi(a)R(a)}{a^4}
  =\prod_{p\le z^2}\sigma_p,\]
  where we now have
\[\sigma_p=1+\sum_{e=1}^\infty\frac{\rho_0(p^e)\varpi(p^e)R(p^e)}{p^{4e}}.\]
Proceeding as before we find that $\sigma_p=1+R(p)/p+O(p^{-3/2})$ when
$p\nmid 2\Dbad$, and $\sigma_p\ll (v+1)^3$ when $p^v\| 2\Dbad$.  This
leads to a bound
\[\Sigma_1\ll\tau(\Dbad)^{O(1)}\mathfrak{S},\]
which is again satisfactory for Theorem \ref{t:nair}, since
\[\varpi(\Dbad)\tau(\Dbad)^{O(1)}h^{2\eta}\ll\Dbad^{\ve}.\]

Finally we must consider $\Sigma_3$. We have
\[\rho_0(a)\ll a^{3+\eta}(\Dbad,a).\]
Moreover
$\varpi(a)\ll a^{\eta}$, whence 
\[\rho_0(a)\varpi(a)a^{-4}\ll (\Dbad,a)a^{-1+2\eta}\le
\Dbad^{5\eta}a^{1-5\eta}.a^{-1+2\eta}\le
\Dbad^{5\eta}a^{-3\eta}(a/z)^{2\eta},\]
since $a\ge z$.  It follows that
\[\Sigma_3
\ll\Dbad^{5\eta}z^{-2\eta}\twosum{z< a\le z^2}{P^+(a)\le\log X}a^{-\eta}
\le\Dbad^{5\eta}z^{-2\eta}\sum_{\substack{a=1\\ P^+(a)\le\log X}}^\infty
a^{-\eta}.\]
The final sum factors as
\[\prod_{p\le\log X}(1-p^{-\eta})^{-1}=
\exp\left\{\sum_{p\le\log X}O(p^{-\eta})\right\}=\exp\{O((\log X)^{1-\eta}))\},\]
which is $O(z^{\eta})$, say. Thus $\Sigma_3\ll\Dbad^{5\eta}z^{-\eta}$,
so that the contribution to (\ref{st}) is satisfactory, provided that
$\eta<\ve/8$. This suffices for the proof of Theorem \ref{t:nair}, 
since we take $\eta=\ve/13$.
\medskip

It remains to prove Lemma \ref{sieve}.  We define 
\[P=\prod_{\substack{p<\tau\\ p\nmid 2a\Dbad}}p.\]
Then 
\begin{align*}
  U(a;\tau)&\leq \#\left\{\x\in \ZZ^4\cap \mathcal{R}: 
  a\mid Q^*(\x), ~ (Q^*(\x),P)=1 \right\}\\
  &\leq \#\left\{\x\in \ZZ^4\cap \mathcal{R}_0: 
  a\mid Q^*(\x), ~ (Q^*(\x),P)=1 \right\},
\end{align*}
where
\[\mathcal{R}_0=\{\x\in \RR^4:  |\x-\u|\leq X\}.\]

We shall use the Selberg sieve, as presented by Halberstam and Richert
\cite[Theorem 4.1]{HR}. We take $\mathcal{A}$ to be the 
sequence of (not necessarily distinct) values $Q^*(\c)/a$, for 
$\c\in\ZZ^4\cap \mathcal{R}_0$, so that we need to understand
\[\#\mathcal{A}_d=
\#\left\{\x\in \ZZ^4\cap \mathcal{R}_0:  ad\mid Q^*(\x)\right\}.\]
When $d<\tau^2$ we have $ad\le Xz^{-11}\tau^2\le Xz^{-7}$. Thus 
the number of
$\x\in \ZZ^4\cap \mathcal{R}_0$ 
in each residue class modulo $ad$ will be
$\meas(\cR_0)(ad)^{-4}+O(X^3(ad)^{-3})$, whence
\[\#\mathcal{A}_d=\meas(\cR_0)\frac{\rho(ad)}{(ad)^4}+O(X^3\rho(ad)(ad)^{-3}).\]
We are only interested in values $d$ which divide $P$.  Hence
$(a,d)=1$, so that
\[\#\mathcal{A}_d=Y\frac{\omega(d)}{d} +R_d\]
with
\[Y=\meas(\cR_0)\frac{\rho(a)}{a^4}, \quad 
\omega(d)=\frac{\rho(d)}{d^{3}} \quad \text{ and }\quad 
R_d\ll \frac{X^3\rho(a)d}{a^{3}}, \]
Here, the last estimate uses the observation that $\rho(d)\le d^4$.

Lemma \ref{lem:rho} yields
\[\frac{\omega(p)}{p}=\frac{\rho(p)}{p^{4}}=
\frac{1}{p}+O\left(\frac{1}{p^2}\right),\]
for any prime $p\nmid 2\Dbad$. Hence 
$\omega$ satisfies the conditions for \cite[Theorem 4.1]{HR} with
$\kappa=1$. It then follows that
\beql{Ss}
U(a;\tau)\ll Y\prod_{p\mid P}\left(1-\frac{\omega(p)}{p}\right)+
\sum_{d<\tau^2}\tau_3(d)X^3\rho(a)a^{-3}d,
\eeq
where the product is 
\[  \prod_{\substack{p<\tau\\ p\nmid 2a\Dbad }}\left(1-\frac{\rho(p)}{p^{4}}\right)
  \ll \prod_{\substack{p<\tau\\ p\nmid 2a\Dbad}}\left(1-\frac{1}{p}\right)\\
  \ll (\log \tau)^{-1}\varpi(\Dbad)\varpi(a),\]
by Mertens' Theorem.  The main term of (\ref{Ss}) is therefore
  \[\ll\varpi(\Dbad)\varpi(a)\frac{X^4\rho(a)}{a^4\log\tau},\]
  while the secondary term is
  \[\ll X^3\rho(a)a^{-3}\tau^5,\]
  say. The main term therefore dominates, since $a\le Xz^{-11}$ and
  $\tau\le z^2$.  This completes the proof of Lemma~\ref{sieve}.

\section{The final stage}

Returning to  \eqref{eq:NB}, we are now ready to conclude our proof of
Theorem~\ref{t:upper}.  Let  
\[S_h(J,K)=
\sum_{\substack{|\c|\le J,\,\, 1\le |Q^*(\c)|\leq K\\ h\mid Q^*(\c)}}R(Q^*(\c))\]
for $J,K\geq 1$.
We divide the available range for $|\c|$ and $|Q^*(\c)|$ into dyadic
intervals, finding that the terms in (\ref{eq:NB}) satisfy
\[S\ll \sum_{\substack{J\ll B^{1/3}\\ K\ll \|Q\|^3J^2}}
\log\left(2+\frac{J^2\|Q\|^3}{K}\right)S_1(J,K)\]
and
\[S_h\ll\sum_{\substack{J\ll B^{1/3}\\ K\ll \|Q\|^3J^2}}J^{-1}
  \log\left(2+\frac{J^2\|Q\|^3}{K}\right)
  S_h(J,K).   \]
where $J$ and $K$ run over powers of 2.

We will bound $S_h(J,K)$ 
by covering the available region for $\c$ by
boxes of side-length $X=B^{1/6}$.  We therefore need to know how many
such boxes are required.
\begin{lemma}
  If $X=B^{1/6}$ and $1\ll J\ll B^{1/3}$ then the region 
  \[|\x|\le J, \;\; |Q^*(\c)|\leq K\]
  can be covered by
\[\ll J^{3/2}\left\{K^{1/2}X^{-1}|\Delta_Q|^{-3/8}+J^{1/4}\right\}\]
  boxes of side $X$.
\end{lemma}

\begin{proof}
Since each such box contains a ball of radius
$X/2$ the number of boxes needed will be at most as large as the
number of balls of radius $X/2$ that are required. By making an
orthonormal change of basis, the problem becomes that of covering the
region
\[\{\x:\,|\x|\le J,\,|D(\x)|\le K\}\]
by balls of radius $[X/2]$, where $D={\rm Diag}(d_1,\ldots,d_4)$
say.  If we arrange the $d_i$ in decreasing order of size we will 
have $|d_1|=\|D\|$. Moreover 
\[d_1\ldots d_4=\det(D)=\det(Q^*)=\Delta_Q^3,\]
whence $|d_1|\ge |\Delta_Q|^{3/4}$. 

If we use balls of radius $X/2$ with centre
$\tfrac18 X\mathbf{n}$, where $\mathbf{n}$ runs over $\ZZ^4$, then they will
cover $\RR^4$.  Moreover, if such a ball overlaps our
region in a point $\x$, then 
we have $\x=\tfrac18 X\mathbf{n}+\y$ for some vector $\y$ with 
$|\y|\le X/2$, so that $|\tfrac18 X\mathbf{n}|\le J+X/2$. Thus 
\[|D(\tfrac18 X\mathbf{n})|=|D(\x-\y)|\le K+X|\mathbf{x}| \|D\|+ 
\frac{X^2}{4}\|D\|.\]
We therefore have to count integer vectors $\mathbf{n}$ for which 
$|\mathbf{n}|\ll J/X+1$ and
\[|D(\mathbf{n})|\ll K/X^2+(J/X)\|D\|+\|D\|.\]
For each choice of $n_2,n_3,n_4$ one has $d_1n_1^2=U+O(V)$,
where
\[U=-d_2n_2^2-d_3n_3^2-d_4n_4^2\]
and
\[V=KX^{-2}+(J/X+1)|d_1|.\]
This condition restricts $n_1$ to an interval of length
$O(\sqrt{V/|d_1|})$, uniformly in $U$.  Since
$|d_1|\ge|\Delta_Q|^{3/4}$, it follows that there are
\[\ll (J/X+1)^3\left\{1+K^{1/2}X^{-1}|\Delta_Q|^{-3/8}+J^{1/2}X^{-1/2}\right\}\]
integer vectors $\mathbf{n}$.  However $J/X+1\ll J^{1/2}$ since
$J\ll B^{1/3}= X^2$, and similarly $1+J^{1/2}X^{-1/2}\ll J^{1/4}$. The lemma
then follows.
\end{proof}

We now wish to apply Theorem \ref{t:nair}, which has the inconvenient
condition (\ref{eq:assume}). Such a condition is typical of such
estimates, but in this instance we can use a trick to handle
situations where $\|Q\|$ is large compared to $B$, so that
(\ref{eq:assume}) may be assumed in the remaining case.
\begin{lemma}\label{lem:pr}
  It suffices to prove Theorem \ref{t:upper} when the entries of
  $\mathbf{M}$ have no common factor.  In the latter case we have
  $N(B)\ll B$ when $\Delta_Q\not=\square$ and $\|Q\|\gg B^{20}$, and
  $N(B)\ll B^2$ when $\Delta_Q=\square$ and $\|Q\|\gg B^{20}$.
\end{lemma}

\begin{proof}
Suppose that Theorem \ref{t:upper} has been proved for primitive forms
$Q$, and suppose that $Q=kQ'$, with $Q'$ primitive.  Then
\begin{align*} 
\varpi(\Delta_{Q'})&\le\varpi(\Delta_Q),\quad\Dbad(Q')=k^{-4}\Dbad(Q),\quad
\frac{\|Q'\|^{5/2}}{|\Delta_{Q'}|^{5/8}}=\frac{\|Q\|^{5/2}}{|\Delta_Q|^{5/8}},\\
\Pi_B(Q')&\le\varpi(k)\Pi_B(Q),\quad \text{ and }\quad 
B^{4/3}+\frac{B^2}{|\Delta_{Q'}|^{1/4}}
\le k\left(B^{4/3}+\frac{B^2}{|\Delta_Q|^{1/4}}\right).\end{align*}
Since $N(B)$ is the same for the two forms $Q'$ and $Q$ we see that if
Theorem~\ref{t:upper} holds for $Q'$ then it holds for $Q$.

When the form $Q$ is primitive we may apply Lemma 3 of Browning and
Heath-Brown \cite{n-2}, which shows that all relevant solutions of
$Q(\x)=0$ will lie on a second quadric surface $Q'(\x)=0$, unless
$\|Q\|\ll B^{20}$.   As already
remarked, the surface $Q(\x)$ will not contain a $\QQ$-line when
$\Delta_Q\not=\square$, so that any component of $Q=Q'=0$ which is
defined over $\QQ$ must have degree at least 2.  We then see that
$N(B)\ll B$, by the work of Walsh \cite{walsh}.  When
$\Delta_Q=\square$ the intersection $Q=Q'=0$ may contain a $\QQ$-line,
and then the contribution to $N(B)$ is $O(B^2)$.
\end{proof}

We plan to apply Theorem \ref{t:nair}
assuming that  
$\|Q\|\ll B^{20}$, say, so that $Q^*(\c)\ll B^{182/3}$
for $|\c|\ll B^{1/3}$.  With $X=B^{1/6}$ 
the condition (\ref{eq:assume}) holds for large
enough $B$, with $A=365$, say.
The theorem bounds $S^{(h)}(X)$ independently of the location of the box
under consideration, so that if $h\le X^{1-\ve}$ we have
\[S\ll S^{(1)}(X)\sum_{\substack{J\ll B^{1/3}\\ K\ll \|Q\|^3J^2}}
\log\left(2+\frac{J^2\|Q\|^3}{K}\right)
J^{3/2}\left\{K^{1/2}X^{-1}|\Delta_Q|^{-3/8}+J^{1/4}\right\},\]
and
\[S_h\ll S^{(h)}(X)\sum_{\substack{J\ll B^{1/3}\\ K\ll \|Q\|^3J^2}}
  \log\left(2+\frac{J^2\|Q\|^3}{K}\right)
J^{1/2}\left\{K^{1/2}X^{-1}|\Delta_Q|^{-3/8}+J^{1/4}\right\},\]
the variables $J$ and $K$ running over powers of 2.

 Since $\Delta_Q\ll\|Q\|^4$ we now find that
 \begin{align*}
 \sum_{ K\ll \|Q\|^3J^2}\log\left(2+\frac{J^2\|Q\|^3}{K}\right)
\left\{K^{1/2}X^{-1}|\Delta_Q|^{-3/8}+J^{1/4}\right\}
&\ll
J\|Q\|^{3/2}X^{-1}|\Delta_Q|^{-3/8}+J^{1/4}(\log B)^2\\
&\ll  \|Q\|^{3/2}|\Delta_Q|^{-3/8}(JX^{-1}+J^{1/4}(\log B)^2)\\
&\ll  \|Q\|^{3/2}|\Delta_Q|^{-3/8}B^{1/6}.
\end{align*}
We may multiply by $J^{3/2}$ and sum over $J$ to find that
\[S\ll S^{(1)}(X)\|Q\|^{3/2}|\Delta_Q|^{-3/8}B^{2/3}.\]
Similarly, we can multiply by $J^{1/2}$ and sum to obtain
\[S_h\ll S^{(h)}(X)\|Q\|^{3/2}|\Delta_Q|^{-3/8}B^{1/3},\]
provided that $h\le X^{1-\ve}$.

We now apply Theorem \ref{t:nair}, whence
\[S\ll \Dbad^{\ve}\|Q\|^{3/2}|\Delta_Q|^{-3/8}B^{2/3}
\mathfrak{S}\frac{X^4}{\log X}\ll
\Dbad^{\ve}\left(\frac{\|Q\|^4}{|\Delta_Q|}\right)^{3/8}B^{4/3}
  \frac{\mathfrak{S}}{\log B},\]
  and
\begin{align*}
B\frac{\|Q\|\Dbad^{1/4}}{|\Delta_Q|^{1/2}}
\max_h\frac{h}{(\Dbad^3,h^4)^{1/4}}S_h
&\ll
B\frac{\|Q\|\Dbad^{1/4}}{|\Delta_Q|^{1/2}}\Dbad^{\ve}
\|Q\|^{3/2}|\Delta_Q|^{-3/8}B^{1/3}\mathfrak{S}\frac{X^4}{\log X}\\
&\ll \Dbad^{1/4+\ve}\left(\frac{\|Q\|^4}{|\Delta_Q|}\right)^{5/8}
\frac{B^2}{|\Delta_Q|^{1/4}}\frac{\mathfrak{S}}{\log B}.
\end{align*}
The condition $h\le X^{1-\ve}$ is satisfied automatically for $h\mid\Dbad^3$,
under the assumption that $\Dbad\le B^{1/20}$.
When we insert these bounds
into (\ref{eq:NB}) we find that the estimate for $S$ dominates $B$,
since $\mathfrak{S}\ge 1$, whence
\[N(B)\ll \left\{\Dbad^{\ve}\left(\frac{\|Q\|^4}{|\Delta_Q|}\right)^{3/8}B^{4/3}
+\Dbad^{1/4+\ve}\left(\frac{\|Q\|^4}{|\Delta_Q|}\right)^{5/8}
\frac{B^2}{|\Delta_Q|^{1/4}}\right\}\frac{\mathfrak{S}}{\log B}.\]

To complete the proof of Theorem \ref{t:upper} all we need do is
estimate $\mathfrak{S}$. When $p\mid 2\Delta_Q$ or $\chi(p)=1$ we have
\[1\le 1+\frac{R(p)}{p}=1+\frac{2}{p}\le \left(1+\frac{1}{p}\right)^2\le \left(1+\frac{1}{p}\right)(1-p^{-1})^{-1},\]
while if $\chi(p)=-1$ we have
\[1= 1+\frac{R(p)}{p}= \left(1-\frac{1}{p}\right)(1-p^{-1})^{-1}.\]
It follows that 
\[\mathfrak{S}\le
(1+1/2)\varpi(\Delta_Q)\prod_{p\le B}\left(1+\frac{\chi(p)}{p}\right)
\prod_{p\le B}(1-p^{-1})^{-1},\]
whence
\[\frac{\mathfrak{S}}{\log B}\ll\varpi(\Delta_Q)\Pi_B,\]
by Mertens' Theorem and the definition \eqref{eq:cQ} of $\Pi_B$. 
This suffices for Theorem~\ref{t:upper}
when $\Delta_Q\neq \square$.

\section{The case of square discriminant}

The preceding argument needs minor modifications when $\Delta_Q$ is a
non-zero square. We will need a number of basic facts from Diophantine
geometry, and will be relatively brief, since the case of non-square
discriminants is the main focus of the paper.
 
Almost all of our argument goes through as
before. Indeed Lemma 
\ref{lem:pr} was already formulated in a way that caters for the
present case. However, at the end
of Section \ref{sP} we can no longer dispose of points on $\QQ$-lines
so readily. We must therefore allow for an additional contribution to $N(B)$
resulting from points which lie on $\QQ$-lines 
$L_1,\ldots,L_N$ contained in the
intersection of the surface $Q=0$ with various
planes $\c.\x=0$.  These planes will have $|\c|\ll B^{1/3}$, and each
such plane can contain at most two such lines.  There are therefore
$N=O(B^{4/3})$ lines to consider.

The integer points on a $\QQ$-line $L$ 
 form a  $2$-dimensional integer sublattice of
determinant $\mathsf{d}(L)$  say.  
A straightforward application of Lemma \ref{lem:dav} shows that the
number of primitive integer points
on $L$ which have height at most $B$ is $O(1+B^2/\mathsf{d}(L))$.
It follows that when $\Delta_Q=\square$ we have an extra
contribution to $N(B)$ of order
\[\sum_{n\le N}\left(1+\frac{B^2}{\mathsf{d}(L_n)}\right).\]

Any $\QQ$-line $L\subset \mathbb{P}^3$ 
corresponds to a rational  point $P_L$ on 
the Grassmannian
$\mathbb{G}(1,3)\subset\mathbb{P}^5$, 
via the familiar Pl\"{u}cker embedding. 
Each $P_L\in \mathbb{G}(1,3)(\QQ)$ has a height 
 $H(P_L)$, which is the Euclidean norm of the corresponding 
primitive integer vector in $\mathbb{P}^5(\QQ)$. Moreover, we have 
$H(P_L)=\mathsf{d}(L)$.  
Consider the subset 
of $P_L\in \mathbb{G}(1,3)$ for which the  line $L$ is contained
in the smooth quadric surface $Q=0$.
According to   Harris \cite[Ex.~6.7]{harris}, this set is the locus of a
smooth conic in $\mathbb{P}^5$. 
But an irreducible  conic in $\mathbb{P}^5$ has
$O(H)$ rational points of height at most $H$
by the work of Walsh \cite{walsh}, 
with an implied constant independent of the conic.
We will write $c$ for the constant occurring here. 
It follows that, for any positive $H$, the number of $\QQ$-lines $L$ 
contained in the surface and having $\mathsf{d}(L)\le H$, is 
at most $cH$. 

Suppose that we have ordered the lines $L_n$ in order of 
non-decreasing height, so that 
$\mathsf{d}(L_n)=h_n$
with $h_1\leq \dots \leq h_N$. 
Taking $H=h_n$ above, we deduce that $n\le ch_n$, since there are at least 
$n$ admissible lines of height up to $h_n$. 
But then $\mathsf{d}(L_n)=h_n\geq c^{-1}n$ for each $n$ and 
it follows that 
\[\sum_{n\le N}\left(1+\frac{B^2}{\mathsf{d}(L_n)}\right)\ll
\sum_{n\le N}\left(1+\frac{B^2}{n}\right)\ll B^{4/3}+B^2\log B.
\]
Thus the extra contribution from the $\QQ$-lines  is $O(B^2\log B)$.

It follows that
\[N(B)\ll_{\ve} B^2\log B+\varpi(\Delta_Q)\Dbad^{1/4+\ve}
\left(\frac{\|Q\|^4}{|\Delta_Q|}\right)^{5/8}
\Pi_B\left(B^{4/3}+\frac{B^2}{|\Delta_Q|^{1/4}}\right),\]
with
\[\Pi_B=\prod_{p\le B}\left(1+\frac{\chi(p)}{p}\right)=
\prod_{p\le B,\,p\nmid\Delta_Q}\left(1+\frac{1}{p}\right)\gg
|\Delta_Q|^{-\ve}\log B.\]
However since $\Delta_Q$ is a square we have
$\Dbad=|\Delta_Q|$, so that
\begin{align*}
\varpi(\Delta_Q)\Dbad^{1/4+\ve}
\left(\frac{\|Q\|^4}{|\Delta_Q|}\right)^{5/8}
\Pi_B\left(B^{4/3}+\frac{B^2}{|\Delta_Q|^{1/4}}\right)&\gg
|\Delta_Q|^{\ve}\left(\frac{\|Q\|^4}{|\Delta_Q|}\right)^{5/8}\Pi_B B^2
\gg  B^2\log B.
\end{align*}
It follows that the term $B^2\log B$ is dominated by the other terms.
This suffices to cover the case in which $\Delta_Q$ is a non-zero square.

\section*{Acknowledgments} 

During the preparation of this  paper the authors were
supported by the NSF under Grant No.\ DMS-1440140,  while  in residence at 
the {\em  Mathematical Sciences Research Institute} in Berkeley, California,
during the Spring 2017 semester.

\bibliographystyle{amsplain}

\begin{dajauthors}
\begin{authorinfo}[tdb]
Tim Browning\\ 
School of Mathematics\\
University of Bristol\\ Bristol\\ BS8 1TW\\ UK\\
and\\ 
IST Austria\\ Am Campus 1\\ 3400 Klosterneuburg\\ Austria\\
t.d.browning@bristol.ac.uk, tdb@ist.ac.at
\end{authorinfo}

\begin{authorinfo}[rhb]
D.R. Heath-Brown\\
Mathematical Institute\\
Radcliffe Observatory Quarter\\ Woodstock Road\\ Oxford\\ OX2 6GG\\ UK\\
rhb@maths.ox.ac.uk
\end{authorinfo}

\end{dajauthors}

\end{document}